\documentclass{paper}[12pt] 

\usepackage{amssymb}   
\usepackage{amsthm}
\usepackage{amsmath}
\usepackage{amsfonts}
\usepackage{latexsym} 
\usepackage{eucal}
\usepackage{indentfirst} 
\usepackage{graphicx} 
\usepackage{verbatim} 

\newtheorem{theorem}{Theorem}[section]

\newtheorem{lemma}[theorem]{Lemma}

\newtheorem{corr}[theorem]{Corollary}
\newtheorem{definition}[theorem]{Definition}

\def\R{{\mathbb R}} 
\def\Q{{\mathbb Q}}

\begin{document}

\overfullrule=0pt
\baselineskip=24pt
\font\tfont= cmbx10 scaled \magstep3
\font\sfont= cmbx10 scaled \magstep2
\font\afont= cmcsc10 scaled \magstep2
\title{\tfont A thermodynamic classification of pairs of real numbers\\
via the Triangle Multi-dimensional continued fraction }
\bigskip
\bigskip
\author{Thomas Garrity\\  
Department of Mathematics and Statistics\\ Williams College\\ Williamstown, MA  01267\\ 
email:tgarrity@williams.edu }

\date{}

\maketitle
\begin{abstract}
 A new classification scheme for pairs of real numbers is given, generalizing work from  \cite{Garrity10}, which in turn was motivated by ideas from statistical mechanics in general and work of Knauf \cite{Knauf1} and Fiala and Kleban \cite{Fiala-Kleban1} in particular. Critical for this classification are the number theoretic and geometric properties of the triangle map, a type of multi-dimensional continued fraction.

\end{abstract}

\section{Introduction}

In \cite{Garrity10}, a new classification scheme for real numbers was given, inspired by ideas from statistical mechanics.  The key technical tools involved continued fractions.  In that paper, it was asked if similar classification schemes occur for various multi-dimensional continued fraction algorithms.  The goal of this paper is to show that such a classification scheme does happen for pairs of real numbers using the triangle map \cite{Messaoudi-Nogueira-Schweiger09} \cite {Assaf05} \cite{Garrity01} 
\cite{Schweiger05}, a type of multi-dimensional continued fraction.

Similar to  \cite{Garrity10},  we use the  thermodynamic formalism (developed by Ruelle \cite{Ruelle2} \cite{Ruelle3} , Sinai \cite{Sinai1}  and others in an attempt to put statistical 
mechanics on a firm mathematical foundation)  to define  a partition function for each pair of real numbers.  Using the real-world interpretation of the partition function, we classify different pairs of reals via their critical point phenomena.

The outline of the method is as follows.  We first define the relevant partition function.  Key will be that each pair $(\alpha, \beta)$ of real numbers gives rise to a different partition function, and hence can be thought of as giving rise to a distinct thermodynamic system.  Next we use the critical point phenomena for each of these pairs $(\alpha, \beta)$ to give the needed definitions for our new classification scheme.  In order for this to be a reasonable classification scheme, we next state the theorems that show different pairs $(\alpha, \beta)$ can be distinguished.  In section four, we put our partition functions into the rhetoric of triangle sequences, which is the multidimensional continued fraction that plays the role of continued fractions from \cite{Garrity10}.  The rest of the paper contains the number theoretic proofs of our main theorems.

There are two interrelated justifications for this paper.  First, it shows that analogs of the work in \cite{Garrity10} apply  for  at least one multidimensional continued fraction algorithm.  This suggests that each multidimensional continued fraction has a corresponding statistical mechanical interpretation.  (For more on the general theory of multidimensional continued fractions, see Schweiger's  {\it Multidimensional Continued Fractions} \cite{Schweiger1}; also, many of these different algorithms have been put into a common framework \cite{DasarathaFlapanGarrityLeeMihailaNeumann-Chun-Peluse-Stoffregen1}, using the triangle map as a starting point.)  Second, the proofs of these generalizations will not be obvious.  The proofs of the two main theorems for this paper use key technical aspects of the triangle map, in particular its underlying geometric motivation, which is why the proofs of these generalizations are not  just exercises.

This current paper and its predecessor  \cite{Garrity10} are certainly not the first to apply statistical mechanical ideas to number theoretic questions.  For example,  there is also the transfer operator method, applied primarily to the Gauss map, which allows, in a natural way, 
tools from functional  analysis to be used. This was pioneered by Mayer  (see his \cite{Mayer2} for a survey), and nontrivially extended by Prellberg \cite{Prellberg1}, by Prellberg and Slawny  \cite{Prellberg-Slawny1}, by Isola \cite{Isola1} and recently by Esposti, Isola and Knauf   \cite{Esposti-Isola-Knauf1}. A good introduction to this work is in chapter nine of Hensley \cite{Hensley1}.

Other links of  statistical mechanics to number theory include the work of Knauf \cite{Knauf2} \cite{Knauf3} \cite{Knauf4}, of  Mend\`es France and Tenenbaum \cite{MendesFrance-Tenenbaum1} \cite{MendesFrance-Tenenbaum2},  of Guerra and Knauf \cite{Guerra-Knauf1}, of Contucci and Knauf \cite{Contucci-Knauf1}, of Fiala, Kleban and \"{O}zl\"{u}k \cite{Fiala-Kleban-Ozluk1}, of Kleban and \"{O}zl\"{u}k \cite{Kelban-Ozluk1}, of  Prellberg, Fiala and Kleban \cite{Prellberg-Fiala-Kleban1}, of Feigenbaum, Procaccia and Tel \cite{Feigenbaum-Procaccia-Tel1}, of Kallies, \"{O}zl\"{u}k, Peter and Snyder    \cite{Kallies-Ozluk-Peter-Syder1}  and others.

\section{The partition function}
\subsection{The General Set-up}

Motivated by  statistical mechanics, we need to find an appropriate partition function.  
In section 3.1 in \cite{Garrity10}, a general framework for developing number theoretic partition functions was given.  Here we will simply state what our desired partition function should be.

For  matrices $A=(a_{ij})$ and $B=(b_{ij})$,  the Hilbert-Schmidt 
product (which is also called the Hadamard product) is
$$A*B=Tr(AB^T)=\sum_{1\leq i,j \leq n}a_{ij}b_{ij}.$$
For example, thinking of a $3\times 3$  matrix as an element of $\R^9$, then $A*B=Tr(AB^T)$ is simply the dot product of the two vectors.

For any $(\alpha, \beta)\in \R^2$, define
$$M=\left( \begin{array}{ccc} 0&0&1\\0&0&\alpha \\0&0& \beta \end{array} \right),$$
and define
$$A_0=\left( \begin{array}{rrr} 0&0&1\\1&0&-1 \\0&1& 0 \end{array} \right), A_1=\left( \begin{array}{rrr} 1&0&0\\0&1&0 \\-1&0& 1\end{array} \right).$$
 In this paper, our state space will be, for each positive integer $N$, 
 
 $${\mathcal S}_N=\{(\sigma_1,\ldots, \sigma_N):\sigma_i=0\; \mbox{or} \;1\}.$$
 For each $N$-tuple $I=(\sigma_1,\ldots, \sigma_N)\in {\mathcal S}_N,$ set
 $$A^I =A_{\sigma_1}A_{\sigma_2} \cdots A_{\sigma_N}.$$
 
 \begin{definition} The partition function $Z_N(\alpha,\beta, s)$ is
 $$Z_N(\alpha,\beta, s) = \sum_{I\in {\mathcal S}_N} \frac{1}{|M*A^I|^s}.$$
 \end{definition}
 By direct calculation, we have 
 $$Z_N(\alpha,\beta, s) = \sum_{I\in {\mathcal S}_N} \frac{1}{\left| (\begin{array}{ccc} 1 & \alpha & \beta \end{array}) A^I \left( \begin{array}{c} 0\\0\\1\end{array} \right)  \right|^s}.$$
 If we denote the third column of the matrix $A^I$ as

 $$A^{I}  \left( \begin{array}{c} 0\\0\\1\end{array} \right) = \left( \begin{array}{c} x(I)\\y(I)\\z(I)\end{array} \right), $$
 then
  $$Z_N(\alpha,\beta, s) = \sum_{I\in {\mathcal S}_N} \frac{1}{\left| x(I) + \alpha y(I) + \beta z(I) \right|^s}.$$
  Thus the terms in $Z_N$ that matter most are those $A^I$ whose third column is close to being perpendicular to the vector $ (\begin{array}{ccc} 1 & \alpha & \beta \end{array})$.  The behavior of $Z_N$ is a measure of how good of  a Diophantine approximation we can achieve by iterations of the matrices $A_0$ and $A_1$.  This will be made clear in section \ref{triangle sequences}   in our discussion of the triangle sequence. 
  
  As discussed in \cite{Garrity10},  if we follow the motivation and inspiration from thermodynamics, we expect the partition function to pack a tremendous amount of information. Also, on a much more minor point, the variable $s$ should represent the inverse of temperature (though in this number theoretic system, there is of course no true notion of temperature for numbers.

  The {\it free energy} is the function 
  $$f(s) = \lim_{N\rightarrow \infty} \frac{\log (Z_N(\alpha, \beta, s)}{N}, $$
  when the limit exists.  From physics, the values of $s$ where $f(s)$ is non-analytic should be important.  In real world examples, it is believed that it is at these points where critical phenomena should occur.  We will distinguish different pairs of real numbers based on how their corresponding number-theoretic free energies behave.

\section{Classifying pairs of reals via free energy}
With the notation above,
we have for each pair of   real numbers $(\alpha, \beta)$ and each positive integer $N$ the partition 
function
$Z_N(\alpha,\beta,s).$

\begin{definition} A  pair of real numbers $(\alpha, \beta)$ has a $k$-free energy limit if 
there is a number $s_c$ such that
$$\lim_{N\rightarrow \infty}\frac{\log(Z_N(\alpha,\beta,s))}{s N^k}$$
exists for  all $s > s_c$
\end{definition}
The subscript  `c'  is used to suggest {\it critical point}.
For $k=1$, this is a number-theoretic version of the free energy of the system.
By an abuse of notation we will also say that $(\alpha, \beta)$ has a $f(N)$-free energy limit, for an increasing function $f(N)$ if
$$\lim_{N\rightarrow \infty}\frac{\log(Z_N(\alpha,\beta,s))}{s f(N)}$$
exists for all all $s > s_c$.

This will only be a reasonable classification scheme if we can show the following two  theorems.


\begin{theorem}\label{badones} For any positive real number $k$, there exists a pair of real numbers $(\alpha, \beta)$ that does not have a  $k$-free 
 energy limit, for any value of $s$.
\end{theorem}

 \begin{theorem}\label{goodones}  Let $(\alpha, \beta)$ be a pair of real numbers  such that there is a positive constant $C$ and constant $d\geq 2$ with 
 $$\frac{1}{Cb^d}\leq |p+ \alpha q + \beta r|$$
 for all relatively prime integers $p$, $q$ and $r,$ with $b$ the maximum of $|p|, |q|$ and $|r|.$.  Then $(\alpha, \beta)$
 has a k-free 
 energy limit for  $s > 2$, for any $k>1$, and in fact, the k-free energy limit is zero.
\end{theorem}
From Theorem 6.4 in \cite{Burger-Tubbs1},  the above yields
\begin{corr} All pairs $(\alpha, \beta)$ of algebraic numbers, each of which has degree at least three and for which $1, \alpha, \beta$ are linearly dependent over $\Q$,   have  k-free 
 energy limits equal to zero, for any $k>1$. 
 \end{corr}
 In particular, 
 \begin{corr} For any algebraic number $\alpha$ of degree at least three, the pair $(\alpha, \alpha^2)$ has k-free 
 energy limits equal to zero, for any $k>1$. 
 \end{corr}


\section{The triangle sequence}\label{triangle sequences}  
\subsection{Review of the triangle sequence}
This is a rapid fire overview of the triangle sequence   \cite{Garrity01}  \cite {Assaf05} \cite{Schweiger05} \cite{Messaoudi-Nogueira-Schweiger09} 
.  

Partition  the triangle $ \bigtriangleup = \{ (x,y): 1 \geq x \geq y >  0 \}$  
 into disjoint triangles
 $$\bigtriangleup_{k} = \{ (x,y) \in \bigtriangleup: 1 - x - ky 
\geq 0 > 1 - x - (k+1)y \},$$
\begin{center}
\setlength{\unitlength}{.1 cm}
\begin{picture}(70,70)
\put(5,5){\line(1,0){60}}
\put(65,5){\line(0,1){60}}
\put(5,5){\line(1,1){60}}
\put(0,0){(0,0)}
\put(60,0){(1,0)}
\put(60,68){(1,1)}

\put(65,5){\line(-1,1){30}}
\put(65,5){\line(-2,1){40}}
\put(65,5){\line(-3,1){45}}
\put(65,5){\line(-4,1){48}}
\put(65,5){\line(-5,1){50}}

\put(5,40){$\triangle_{0}$}
\put(12,40){\vector(1,0){35}}
\put(5,27){$\triangle_{1}$}
\put(12,27){\vector(1,0){22}}
\put(5,22){$\triangle_{2}$}
\put(12,22){\vector(1,0){15}}
\put(5,18){$\triangle_{3}$}
\put(12,18){\vector(1,0){12}}
\put(5,14){$\triangle_{4}$}
\put(12,14){\vector(1,0){11}}
\end{picture}
\end{center}

For a pair
of real numbers $(\alpha, \beta)\in \triangle_k$ define 
$$T(\alpha,\beta) =\left (\frac{\beta}{\alpha}, \frac{1 - \alpha - k 
\beta}{\alpha}\right).$$ The {\it triangle sequence}
  for a 
pair $(\alpha, \beta)$ is the infinite sequence of nonnegative integers 
$(a_0,a_1,a_2, \ldots )$, where  $T^k(\alpha, \beta)\in \bigtriangleup_{a_k}$.  The map $T$ is a one-to-one onto map from each subtriangle $\triangle_k$ to the original $\triangle.$ In \cite{Messaoudi-Nogueira-Schweiger09}, Messaoudi, Nogueira and Schweiger showed that this is an ergodic map.  
 The triangle sequence 
is said to terminate at step $k$ if $T^k(\alpha, \beta)$ lands on the 
interval $\{(t,0):0\leq t\leq 1\}$.

Now to start linking the triangle sequence with our partition function.  
Another way of thinking about triangle sequences is as a method for 
producing integer lattice vectors in space
that approximate the plane $x+\alpha y + \beta z =0.$   Set 
$$C_{-3} = \left( \begin{array}{c} 1\\0\\0\end{array}  \right)  , C_{-2} =\left( \begin{array}{c} 0\\1\\0\end{array}  \right) ,C_{-1} =
\left( \begin{array}{c} 0\\0\\1\end{array}  \right) .$$
If the triangle sequence for $(\alpha, \beta)$ is $(a_0,a_1,a_2, \ldots )$, set
$$C_k = C_{k-3}-C_{k-2} - a_k C_{k-1}.$$
The triangle sequence can  be defined in terms of the dot products 
$d_k=(1,\alpha, \beta)\cdot C_k.$
  Assuming we know the number $a_0,\ldots,a_k$, then $a_{k+1}$
 is the nonnegative integer such that
$$d_{k-2}-d_{k-1}-a_{k+1}d_k \geq 0 > d_{k-2}-d_{k-1}-(a_{k+1}+1)d_k .$$
Then
$$d_{k+1}= d_{k-2}-d_{k-1}-a_{k+1}d_k.$$
For each sequence $(a_0,a_1,a_2, \ldots, a_k ),$ there is a subtriangle $\triangle(a_0,\ldots , a_k)$ of $\triangle$ such that 
$T^{(k+1)}:\triangle(a_0,\ldots , a_k) \rightarrow \triangle$ is one-to-one and onto.   In \cite{Assaf05}, the vertices of this triangle were determined.  For this, we will need a little notation.
Recalling  the notation from \cite{Assaf05}, for vectors
$$T= \left( \begin{array}{c} a\\ b\\ c   \end{array}\right) \; \mbox{and} \; S=  \left( \begin{array}{c} d\\ e\\ f   \end{array}\right) $$
with $a, d, a+d\neq 0,$
define
$$\hat{T}=  \left( \begin{array}{c} \frac{b}{a} \\ \frac{c}{a} \end{array}\right)$$
and further, define
$$T\hat{+}S= \widehat{T+S} =  \left( \begin{array}{c} \frac{b+ e}{a+d}\\ \frac{c+ f}{a+d}   \end{array}\right) .$$
(Such a sum is  called a {\it Farey sum}.)  Further, continuing with this type of notation, we define
$$ \left( \begin{array}{c} a\\ b\\ c   \end{array}\right) \hat{+}  \left( \begin{array}{c} d\\ e\\ f   \end{array}\right) \hat{+}   \left( \begin{array}{c} g\\ h\\ f  \end{array}\right) =   \left( \begin{array}{c} \frac{b+ e+h}{a+d+g}\\ \frac{c+ f+i}{a+d+g}   \end{array}\right) .$$

As shown in  \cite{Assaf05}, we can now cleanly describe the vertices for the triangle $\bigtriangleup (a_0, \ldots, a_n)$.  Set
$$X_n=\left(\begin{array}{c} x_n \\ y_n \\ z_n \end{array}  \right)=C_n\times C_{n+1}.$$
Then
the vertices for the triangle $\bigtriangleup (a_0, \ldots, a_n)$ are
$\hat{X}_{n-1}$, $\hat{X}_{n}$ and $ X_n \hat{+}X_{n-2}$.
At a critical stage for our argument, we will be using that 
$$ \triangle (\hat{\alpha},\hat{X}_{k-1},\hat{X}_k) \subset \triangle (X_k  \hat{+}a_{k+1}X_{k-1}\hat{+} X_{k-2},\hat{X}_{k-1},\hat{X}_k), $$
which will mean that 
$ \mbox{area}\; \triangle (\hat{\alpha},\hat{X}_{k-1},\hat{X_k}) < \mbox{area}\; \triangle (X_k\hat{+}a_{k+1}X_{k-1}\hat{+} X_{k-2},\hat{X}_{k-1},\hat{X}_k) ,$ as seen in 

\resizebox{3 in}{3 in}{\includegraphics[]{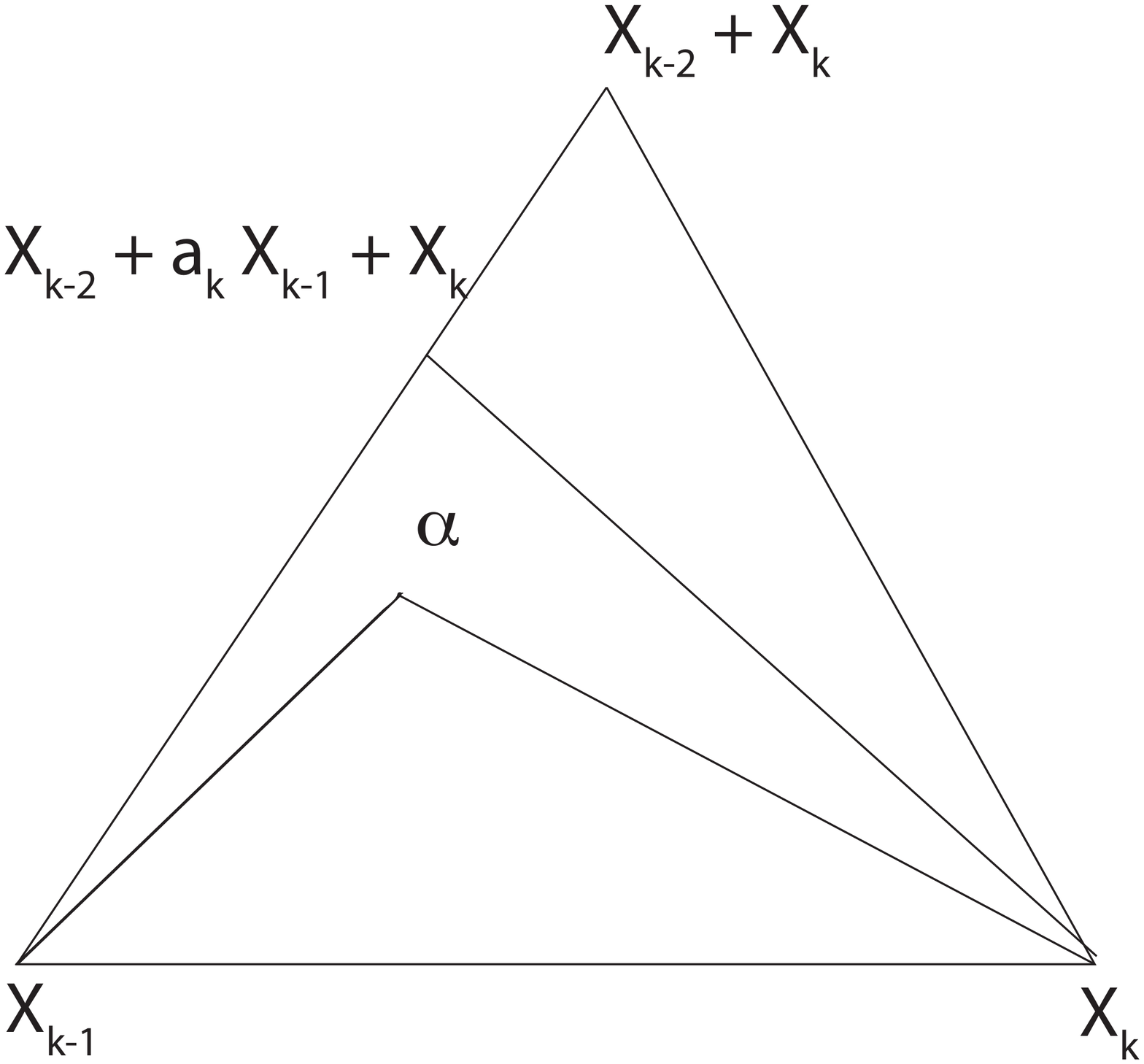}}

\subsection{The triangle sequence and the $A_0$ and $A_1$}
Here we provide the key link between the geometry of the triangle sequence and the terms for our partition function.
Let $(\alpha, \beta) \in \triangle$ have triangle sequence $(a_1,a_2, \ldots ).$  For each $k$, set 
$$N_k=(a_1+1) + \ldots (a_k+1).$$
Then, as seen in  \cite {Assaf05}, 
the matrix $A_1^{a_1}A_0 A^{a_2}A_0 \cdots A^{a_k}A_0$ will have columns $C_{N_k-2},C_{N_k-1},C_{N_k}.$
Thus for the sequence $I=(1^{a_1},0,1^{a_2},0, \ldots , 1^{a_k},0),$ we have
$$M*A^I = d_{N_k}.$$

\section{Proof of Theorem \ref{badones}}

Using the notation of the previous two sections, we start with
\begin{lemma} For all $k$,
$$|\left( \begin{array}{ccc} 1 & \alpha & \beta)\end{array} \right) \cdot C_k| \leq \frac{1}{|x_{k+1}|}.$$
\end{lemma}

\begin{proof}
We know that 
$$\hat{\alpha} \in \triangle (X_k\hat{+}X_{k-2},\hat{X}_{k-1},\hat{X}_k).$$
The definition of $X_{k+1}$ and of $a_{k+1}$ is that 
$$  \left(\begin{array}{ccc}  \alpha & \beta\end{array} \right) \in \triangle (X_k\hat{+}a_{k+1}X_{k-1}\hat{+} X_{k-2},\hat{X}_{k-1},\hat{X}_k)$$
but that
$$ \left(\begin{array}{cc}  \alpha & \beta\end{array} \right)\not\in \triangle (X_k\hat{+}(a_{k+1}+1)X_{k-1}\hat{+} X_{k-2},\hat{X}_{k-1},\hat{X}_k).$$
Then we have 
\begin{eqnarray*}
| \left(\begin{array}{ccc} 1 & \alpha & \beta\end{array} \right) \cdot C_k)| &=& | \left(\begin{array}{ccc} 1& \alpha & \beta \end{array} \right) \cdot (X_{k-1}\times X_k) | \\
&=& | \det ( \left(\begin{array}{ccc} 1& \alpha & \beta\end{array} \right), X_{k-1},X_{k})| \\
&=& |x_{k-1}x_k \mbox{area} \triangle ( \left(\begin{array}{cc}  \alpha & \beta\end{array} \right),\hat{X}_{k-1},\hat{X}_k) |  \\
&\leq&  |x_{k-1}x_k \mbox{area} \triangle (X_k\hat{+}a_{k+1}X_{k-1}\hat{+} X_{k-2},\hat{X}_{k-1},\hat{X}_k) | \\
&=& \left|x_{k-1}x_k  \left(\frac{\det (X_k+X_{k-2},X_{k-1},X_k)  }{x_{k-1}x_k (x_k + a_{k+1}x_{k-1}+x_{k-2})} \right)  \right|  \\
&=& \frac{1}{|x_k + a_{k+1}x_{k-1}+x_{k-2}|}\\
 &=& \frac{1}{|x_{k+1}|}.
 \end{eqnarray*}
as desired.

Here we are using the following two facts.  First, if
$$A= \left( \begin{array}{ccc} a_{11}& a_{12} & a_{13} \\
a_{21}& a_{22} & a_{23} \\
a_{31}& a_{32} & a_{33}    \end{array}  \right),$$
the area of the triangle with vertices
$$\left( \begin{array}{c} \frac{ a_{21} }{ a_{11}} \\  \frac{ a_{31} }{ a_{11}}\end{array}
\right), \left( \begin{array}{c} \frac{ a_{22} }{ a_{12}} \\  \frac{ a_{32} }{ a_{12}}\end{array}\right), \left( \begin{array}{c} \frac{ a_{23} }{ a_{13}} \\  \frac{ a_{33} }{ a_{13}}\end{array}
\right)$$
is
$$\frac{\det(A)}{a_{11} a_{12} a_{13}}.$$
Second, we have 
$$\det (X_k+X_{k-2},X_{k-1},X_k) =1,$$
which is stems from the fact that the matrix $(X_k+X_{k-2},X_{k-1},X_k) $ is the product of matrices of determinant one.

\end{proof}

We are now ready to prove Theorem \ref{badones}, namely given any positive real number $k$, there exists a pair of real numbers $(\alpha, \beta)$ that does not have a  $k$-free 
 energy limit, for any value of $s$.

  Let $f(n)$ be a function that strictly increases to infinity.  Choose a sequence of integers $a_1, a_2, \ldots$ inductively by setting
$$a_{m+1} > e^{f(m+1)(a_1 + \ldots a_m + m)^k},$$
with the initial $a_1$ any positive integer greater than two.

Choose a pair $(\alpha, \beta)$ so that the pair's triangle sequence is $(a_1,a_2, \ldots).$
Recall that we set 
$N_m = a_1 + a_2 + \cdots + a_m + m.$

Consider  $Z_{N_m}(\alpha, \beta, s).$  We have the sequence $I=(1^{a_1},0,1^{a_2},0, \ldots , 1^{a_m},0)$ being part of what is summed over for the partition function $Z_{N_m}(\alpha, \beta, s).$  Thus,  for this $I$, we have
$$\frac{1}{|\left( \begin{array}{ccc} 1 & \alpha & \beta)\end{array} \right) \cdot C_{N_m}|^s }= \frac{1}{|M*A^I|^s} <  \sum_{I\in {\mathcal S}_N} \frac{1}{|M*A^I|^s}=Z_{N_m}(\alpha,\beta, s) .$$
Thus we have 
$$|x_{m+1}|^s < Z_{N_m}(\alpha,\beta, s).$$
We will show that 
$$\lim_{m\rightarrow \infty} \frac{\log |x_{m+1}|^s}{N_m^k} = \infty$$
for any value of $s>0.$

We have 
$$x_{m+1} = x_m+a_{m+1}x_{m-1}+ x_{m-2}.$$
Thus if $m$ is even we have 
$$x_{m+1} > a_{m+1}a_{m-1} \cdots a_1$$
and if $m$ is odd we have 
$$x_{m+1} > a_{m+1}a_{m-1} \cdots a_2.$$
In particular, we have 
$$x_{m+1} > a_{m+1}.$$
Then 
\begin{eqnarray*}
\lim_{m\rightarrow \infty} \frac{\log |x_{m+1}|^s}{N_m^k}  &>& \lim_{m\rightarrow \infty} \frac{\log |a_{m+1}|^s}{(a_1+ \cdots a_m + m)^k} \\
&\geq & \lim_{m\rightarrow \infty} \frac{s\log |e^{f(m+1)(a_1 + \ldots a_m + m)^k}|}{(a_1+ \cdots a_m + m)^k}\\
&=& \lim_{m\rightarrow \infty} \frac{sf(m+1)(a_1 + \ldots a_m + m)^k}{(a_1+ \cdots a_m + m)^k}\\
&=& \infty,
\end{eqnarray*}
finishing the proof of Theorem \ref{badones}.

\section{Proof of Theorem \ref{goodones}}

By assumption, $(\alpha, \beta)$ is  a pair of real numbers  such that there is a positive constant $C$ and constant $d\geq 2$ with 
 $$\frac{1}{Cb^d}\leq |p+ \alpha q + \beta r|$$
 for all relatively prime integers $p$, $q$ and $r,$ with $b$ the maximum of $|p|, |q|$ and $|r|.$. We must show that $(\alpha, \beta)$
 has a k-free 
 energy limit for  $s > 2$, for any $k>1$.  In fact, we will see that the k-free energy limit is zero.

 By inverting and raising both sides of the above inequality by $s$, we have
 $$\frac{1}{|p+ \alpha q + \beta r|^s}  \leq C^s b^{sd}.$$
 In this case
 \begin{eqnarray*}
 Z_N(\alpha,\beta, s) &=& \sum_{I\in {\mathcal S}_N} \frac{1}{\left| x(I) + \alpha y(I) + \beta z(I) \right|^s}\\
 &\leq& \sum_{I\in {\mathcal S}_N} C^s (\rm{max}(|x(I),y(I),z(I))^{sd}.
 \end{eqnarray*}
Hence we need to control the growth rates of the $|x(I)|,|y(I)|$ and $|z(I)| $ in terms of the integer $N$.
 We will find fairly crude but workable bounds.  First off change  the minus ones in the matrices $A_0$ and $A_1$ to ones.  Then the growth rates of the $|x(I)|,|y(I)|$ and $|z(I)| $  will be bounded by the $N^th$ Fibanacci number, which is well known to be 
 $$F_N = \frac{\phi^N -\psi^N}{\sqrt{5}},$$
where 
$$\phi = \frac{1+\sqrt{5}}{2}, \;\psi =  \frac{1-\sqrt{5}}{2}.$$ 
Thus 
$$Z_N(\alpha,\beta, s)  \leq 2^N C^s F_N^{sd} < 2^NC^s \left(  \frac{\phi^n }{\sqrt{5}} \right)^{sd}.$$
 
Hence
 \begin{eqnarray*}
 \lim_{N\rightarrow \infty}\frac{\log(Z_N(\alpha;s))}{\beta N^k}  &\leq& \lim_{N\rightarrow \infty}\frac{\log\left(2^N C^s\left(  \frac{\phi^N }{\sqrt{5}} \right)^{sd}\right)}{s N^k} \\
 &=& \lim_{N\rightarrow \infty} \frac{ N\log(2) + s\log(C) - sd\log(\sqrt{5}) + sdN\log(\phi)}{sN^k} \\
 &=& 0,
 \end{eqnarray*}
 for $K>1$, as desired.

\section{Conclusion}To some extent, this paper is the second (with \cite{Garrity10} the first) in a possible series of paper using the thermodynamic formalism to understand more about the structure of the real numbers.  There are many questions left.

Probably the most straightforward, but still non-trivial, ones are to try to mimic the results in this paper for other types of multidimensional continued fractions.  These analogs would involve choosing different matrices $A_0$ and $A_1$.

Our partition function  $Z_N(\alpha, \beta,s)$ depended on the choice of the three-by-three matrix $M$.  Different choices of $M$ give rise to different thermodynamics.  Certainly choosing different $M$'s in the two-by-two case led to the differences in the work of  Knauf \cite{Knauf1}\cite{Knauf2} \cite{Knauf3} \cite{Knauf4}\cite{Knauf5}     and of Fiala, Kleban and Ozluk  \cite{Fiala-Kleban1}\cite{Fiala-Kleban-Ozluk1} versus the work in \cite{Garrity10}.  Are there analogs here?

 We have shown that  some pairs of real numbers have 1-free energy limits while others do not.  This too is just a beginning.  What types of limits are possible?  Are there pairs $(\alpha, \beta)$ for which the sequence $\frac{\log\left(Z_N(\alpha,\beta,s)\right)}{N}$ has any possible limit behavior?  For example, can we  rig $(\alpha, \beta)$ so that any number can be the limit of the sequence.  In fact, we should be able to find such sequences with two accumulation points, three accumulation points, etc.  All of these should provide information about the pair of  real numbers $\alpha$ and $\beta$.

 In \cite{Garrity10}, continued fractions were critical.  Since an amazing amount is known about continued fractions, it is not surprising that there are more results in \cite{Garrity10}.  For example, in that paper it was shown that $e-1$ has a one-free energy limit.  Are there similar such results possible for concrete pairs of reals of the thermodynamic systems of this paper.  These types of questions are to some extent the most intriguing.
 
 Finally, there is the question of putting these results into the language of transfer operators. (See \cite{Ruelle1}, \cite{Mayer1}, \cite{Baladi1} for general references.)


\begin{thebibliography}{99}

\bibitem{Assaf05} S. Assaf, L. Chen, T. Cheslack-Postava, B. Cooper, A. Diesl, T. Garrity, M. Lepinski, A. Schuyler, A Dual Approach to Triangle Sequences: A Multidimensional Continued Fraction Algorithm, {\it Integers}, 5, (2005).





\bibitem{Baladi1} V. Baladi, {\it Positive Transfer Operators and Decay of Correlations }(Advanced Series in Nonlinear Dynamics, Volume 16), World Scientific, 2000.





\bibitem{Burger-Tubbs1} E. Burger and R. Tubbs, {\it Making Transcendence 
Transparent: An intuitive approach to classical transcendental number 
theory}, Springer, 2004.

\bibitem{Contucci-Knauf1} P. Contucci and A. Knauf, The phase transition in statistical models defined on Farey fractions, {\it Forem Math} 9 (1997), pp.547-567.

\bibitem{DasarathaFlapanGarrityLeeMihailaNeumann-Chun-Peluse-Stoffregen1}  K. Dasaratha, L Flapan, T. Garrity, C. Lee, C Mihaila, N. Neumann-Chun, S. Peluse and M. Stoffrege, A Generalized Family of Multidimensional Continued Fractions, in preparation.




\bibitem{Esposti-Isola-Knauf1}  M. Espoti, S. Isola and A. Knauf, Generalized Farey Trees, Transfer Operators and Phase Transition,  {\it Communications 
in Mathematical Physics}, 275 (2007), pp. 298-329.


\bibitem{Feigenbaum-Procaccia-Tel1} M. Feigenbaum, I. Procaccia and T. Tel, Scaling properties of multifractals as an eigenvalue problem, {\it Physical Review A} 39 (1989), pp. 5359-5372.

\bibitem{Fiala-Kleban1} J. Fiala and P. Kleban, Generalized Number Theoretic 
Spin-Chain Conditions to Dynamical Systems and Expectation Values, {\it J. Stat. 
Phys.}, 121 (2005), pp. 553-577.

\bibitem{Fiala-Kleban-Ozluk1} J. Fiala, P. Kleban and A. 
\"{O}zl\"{u}k, The Phase Transition in Statistical Models Defined on Farey 
Fractions, {\it J. Stat. Phys.}, 110 (2003), pp. 73-86.



\bibitem{Garrity01} T. Garrity, On periodic sequences for algebraic numbers, {\it Journal of Number Theory}, 88 (2001), pp. 86-103.

\bibitem{Garrity10} T. Garrity, A thermodynamic classification of real numbers, {\it Journal of Number Theory}, 130 (2010), pp. 1537-1559.





\bibitem{Guerra-Knauf1}  F. Guerra and A. Knauf, Free energy and correlations of the number-theoretical spin chains, {\it J. Math. Physics}, 39 (1998), pp. 3188-3202.

\bibitem{Hensley1}  D. Hensley, {\it Continued Fractions}. World Scientific, 2006.

\bibitem{Isola1} S. Isola, On the spectrum of Farey and Gauss maps, {\it Nonlinearity} 15 (2002), pp. 1521-1539.

\bibitem{Kallies-Ozluk-Peter-Syder1} J. Kallies, A. \"{O}zl\"{u}k, M. Peter, C. Snyder, On asymptotic properties of a number theoretic function arising out of a spin chain model in statistical mechanics. {\it Comm. Math. Phys.} 222 (2001), no. 1, pp. 9-43.


\bibitem{Kelban-Ozluk1} P. Kleban and A. Ozluk, A farey fraction spin chain, {\it Communications 
in Mathematical Physics},203 (1999), pp. 635-647.


\bibitem{Knauf1}  A. Knauf, On a Ferromagnetic Chain, {\it Communications 
in Mathematical Physics}, 153 (1993), pp. 77-115.

\bibitem{Knauf5} A. Knauf Phases of the Number-Theoretical SPin Chain, {\it J. Stat. Phys.} 73, 423-431, 1993.

\bibitem{Knauf2} A. Knauf, On a ferromagnetic spin chain.part ii: thermodynamic limit, {\it J. Math. Physics}, 35 (1994), pp. 228-236.


\bibitem{Knauf3} A. Knauf, The number-theoretic spin chain and the Riemann zeros,  {\it Communications 
in Mathematical Physics}, 196 (1998), pp.703-731.

\bibitem{Knauf4} A. Knauf, Number theory, dynamical systems and statistical mechanics, {\it Rev Math. Phys.} 11 (1998), pp. 1027-1060.



\bibitem{Mayer1}  D. Mayer, {\it The Ruelle-Araki transfer operator in classical statistical mechanics}, Springer, 1980.

\bibitem{Mayer2} D. Mayer, Continued fractions and related transformations, {\it  Ergodic theory, symbolic dynamics, and hyperbolic spaces} (Trieste, 1989), 175--222, 
Oxford Univ. Press (1991), pp. 175-222.

\bibitem{MendesFrance-Tenenbaum1} M. Mend\`es France and G. Tenenbaum. A One-dimensional Model with Phase Transition, {\it Communications in Mathematical Physics}, 154 (1993), pp. 603-611.

\bibitem{MendesFrance-Tenenbaum2}  M. Mend\`es France and G. Tenenbaum, Phase Transitions and Divisors, {\it Probability Theory and Mathematical Statistics}, (Vilnius, 1993), pp. 541-552.



\bibitem{Messaoudi-Nogueira-Schweiger09}  A. Messaoudi, A. Nogueira and F. Schweiger, Ergodic properties of triangle partitions, {\it Monatsh. Math.} 157 (2009), no. 3, pp. 283Ð299. 



\bibitem{Prellberg1} T. Prellberg, Towards a complete determination of the spectrum of the transfer operator associated with intermittency, {\it J. Phys. A} 36 (2003), pp. 2455-2461.








\bibitem{Prellberg-Fiala-Kleban1}  T. Prellberg, J. Fiala and P. Kleban, Cluster Approximation for the Farey Fraction Spin Chain {\it J. Stat. Phys.} 123 (2006), pp. 455-471.

\bibitem{Prellberg-Slawny1} T. Prellberg and J. Slawny, Maps of intervals with indiffferent fixed points: thermodynamic formalism and phase transition, {\it J. Stat. Phys.} 66 (1992), pp. 503-514.

\bibitem{Rockett-Szusz1} A. Rockett and P. Sz\"{u}sz, {\it Continued Fractions}, World Scientific, 1992.


\bibitem{Ruelle1} D. Ruelle, Dynamical zeta functions for maps of the interval, {\it Bull. Amer. Math. Soc.}(NS) 30 (1994), pp. 212-214.

\bibitem{Ruelle2} D. Ruelle, {\it Statistical Mechanics: Rigorous Results}, World Scientific, 1999.

\bibitem{Ruelle3} D. Ruelle, {\it Thermodynamic Formalism: The Mathematical Structure of Equilibrium Statistical Mechanics }, second edition, Cambridge, 2004.




\bibitem{Schweiger1} F. Schweiger, {\it Multidimensional Continued Fractions}, Oxford University Press, 2000.


\bibitem{Schweiger05} F. Schweiger, Periodic multiplicative algorithms of Selmer type,
{\it Integers} 5 (2005), no. 1, A28.




\bibitem{Sinai1} Y. Sinai, {\it Theory of Phase Transitions}, Pergamon, 1983.





			








	
				

	





\end{thebibliography}
\end{document}